\pdfoutput=1

\documentclass{amsart}

\usepackage{verbatim}
\usepackage{xcolor}
\usepackage{tikz}
\usepackage{tikz-cd}
\usepackage{arydshln}
\usepackage{caption}
\usepackage{subcaption}
\usepackage{graphicx}
\usepackage[hyphens]{url}
\usepackage{amsmath}
\usepackage{bbm}
\usepackage{hyperref}
\usepackage{mathrsfs}
\usepackage{gensymb}

\newtheorem{theorem}{Theorem}[section]
\newtheorem{proposition}[theorem]{Proposition}

\newtheorem{corollary}[theorem]{Corollary}

\theoremstyle{definition}

\theoremstyle{remark}
\newtheorem{remark}[theorem]{Remark}

\numberwithin{equation}{section}

\newcommand{\mcg}{\mathrm{Mod}}

\newcommand{\mf}{\mathcal{MF}}

\newcounter{count}

\newcounter{counterk} 
\newcounter{counterl} 
\newcounter{counterc} 
\newcounter{countercc} 
\newcounter{countere} 
\newcounter{counterd} 
\newcounter{countern} 
\newcounter{counterr}
\begin{document}

\title[Effective count of square-tiled surfaces]{Effective count of square-tiled surfaces with prescribed real and imaginary foliations in connected components of strata}

\author{Francisco Arana--Herrera}
\address{Institute for Advanced Study, 1 Einstein Drive, Princeton, NJ 08540, USA.}
\email{farana@ias.edu}

\begin{abstract}
	We prove an effective estimate with a power saving error term for the number of square-tiled surfaces in a connected component of a stratum of quadratic differentials whose vertical and horizontal foliations belong to prescribed mapping class group orbits and which have at most $L$ squares. This result strengthens asymptotic counting formulas in work of Delecroix, Goujard, Zograf, Zorich, and the author.
\end{abstract}

\maketitle


\thispagestyle{empty}

\tableofcontents

\section{Introduction}

In her thesis \cite{Mir04}, Mirzakhani proved asymptotic formulas for the number of simple closed geodesics of a given topological type and length at most $L$ on an arbitrary complete, finite area hyperbolic surface. Inspired by this work, Delecroix, Goujard, Zograf, and Zorich \cite{DGZZ21} proved asymptotic formulas for the number of square-tiled surfaces whose vertical and horizontal foliations belong to prescribed mapping class group orbits and which have at most $L$ squares. Later work of the author \cite{Ara20d} established a direct connection between these two results, obtaining a new proof of the results in \cite{DGZZ21} as a direct consequence of \cite{Mir04}.

Despite the success of these different approaches, the techniques used do not yield effective error terms; indeed, they rely crucially on ergodicity. The main goal of this paper is to prove an effective estimate with a power saving error term for the number of square-tiled surfaces in a connected component of a stratum of quadratic differentials whose vertical and horizontal foliations belong to prescribed mapping class group orbits and which have at most $L$ squares. Let us highlight that this new result not only provides an effective error term, but also applies to connected components of strata that are not the principal stratum.

The proof of this result is based on a novel combination of the connections established in \cite{Ara20d} and work of Eskin, Mirzakhani, and Mohammadi \cite{EMM19}. In \cite{Ara20d}, using work of Hubbard and Masur \cite{HM79}, the author established a direct connection between counting problems of square-tiled surfaces and simple closed multi-curves. In \cite{EMM19}, Eskin, Mirzakhani, and Mohammadi proved an effective estimate for the number of simple closed curves of length at most $L$ on an arbitrary compact surface equipped with a Riemannian metric of negative curvature. To prove this result, a sophisticated theory for counting mapping class group orbits of integral simple closed multi-curves in train track coordinates was developed. New variants of this theory, in conjunction with the connections developed in \cite{Ara20d}, are the main tools used in the proof.

The proof of the main result of this paper reinforces the advantage of exploiting the connection between counting problems of square-tiled surfaces and simple closed multi-curves first introduced in \cite{Ara20d}. The main result also continues a program for proving effective counting results for surfaces, their Teichmüller spaces, and their moduli spaces \cite{EMM19,Ara20b,Ara20c,Ara20e}.

\subsection*{Main result.} A square-tiled surface is a connected, oriented surface constructed from finitely many disjoint unit area squares on the complex plane, with sides parallel to the real and imaginary axes, by identifying pairs of sides by translation and/or $180\degree$ rotation. A square-tiled surface represents a particular Riemann surface together with a quadratic differential obtained by lifting $\mathrm{d}z^2$ from the complex plane. The horizontal core multi-curve of a square tiled-surface is the integrally weighted simple closed multi-curve obtained by concatenating the horizontal segments running through the middle of each square. The vertical core multi-curve of a square tiled-surface is defined in an analogous way. See Figure \ref{fig:example} for an example.

\begin{figure}[h]
	\centering
	\begin{subfigure}[b]{0.38\textwidth}
		\centering
		\includegraphics[width=0.6\textwidth]{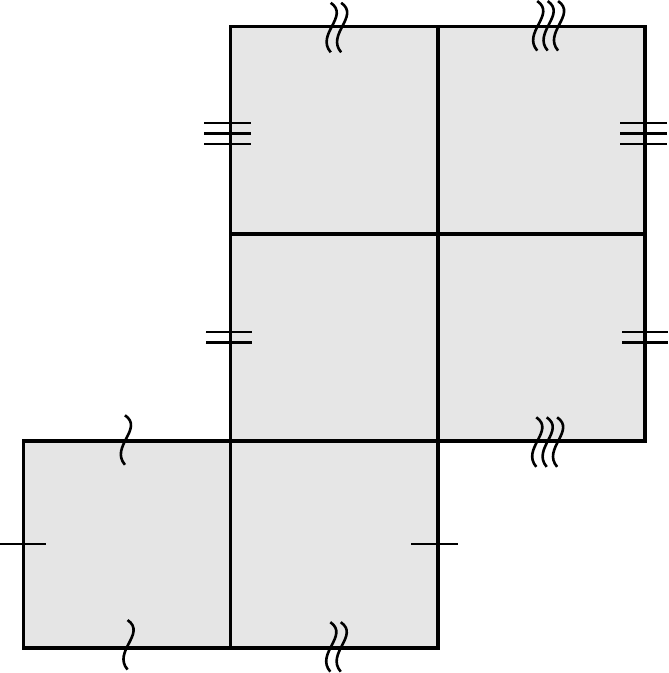}
		\caption{Square-tiled surface.}
		\label{fig:sq_tiled}
	\end{subfigure}
	\quad \quad \quad \\
	~ %add desired spacing between images, e. g. ~, \quad, \qquad, \hfill etc. 
	\begin{subfigure}[b]{0.4\textwidth}
		\centering
		\includegraphics[width=0.6\textwidth]{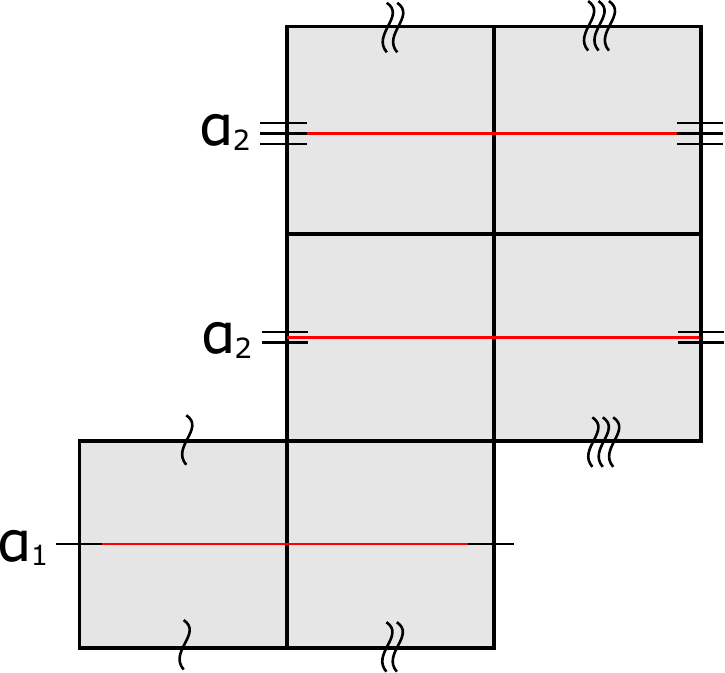}
		\caption{Horizontal core multi-curve.}
		\label{fig:horizontal}
	\end{subfigure}
	\quad \quad \quad
	~ %add desired spacing between images, e. g. ~, \quad, \qquad, \hfill etc. 
	\begin{subfigure}[b]{0.4\textwidth}
		\centering
		\includegraphics[width=0.6\textwidth]{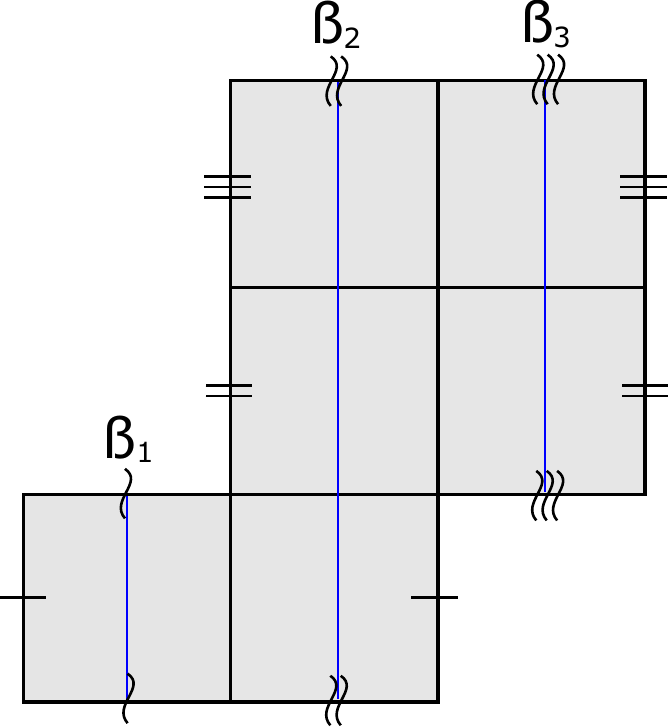}
		\caption{Vertical core multi-curve.}
		\label{fig:vertical}
	\end{subfigure}
	\caption{Example of a square-tiled surface of genus $2$ with two zeroes of order $2$. The horizontal core multi-curve is $\alpha_1 + 2 \alpha_2$. The vertical core multi-curve is $\beta_1 + \beta_2 + \beta_3$.} \label{fig:example}
\end{figure}

Two integral simple closed multi-curves on homeomorphic surfaces are said to have the same topological type if there exists a homeomorphism between the surfaces mapping one multi-curve to the other preserving the weights. 

Let $\mathcal{Q}$ be a connected component of a stratum of quadratic differentials and $\gamma_1$, $\gamma_2$ be a pair of integral simple closed multi-curve on the corresponding topological surface. For every $L > 0$ consider the counting function
\[
sq(\gamma_1,\gamma_2,\mathcal{Q},L) := \#
\left\lbrace 
\begin{array}{c}
\text{square-tiled surfaces in $\mathcal{Q}$ with vertical core }\\
\text{multi-curve of the same topological type as $\gamma_1$,} \\
\text{horizontal core multi-curve of the same}\\
\text{topological type as $\gamma_2$, and $\leq L$ squares}
\end{array}
\right\rbrace /\sim,
\]
where $\sim$ denotes the equivalence relation induced by cut and paste operations.

The following is the main result of this paper.

\begin{theorem}
	\label{theo:main}
	Let $\mathcal{Q}$ be a connected component of a stratum of quadratic differentials of complex dimension $h > 0$ and $\gamma_1$, $\gamma_2$ be a pair of integral simple closed multi-curves on the corresponding topological surface. Then, there exist positive constants $v(\gamma_1,\mathcal{Q}) > 0$, $v(\gamma_2,\mathcal{Q}) > 0$, and $\kappa = \kappa(\mathcal{Q}) > 0$ such that for every $L \geq 0$,
	\[
	sq(\gamma_1,\gamma_2,\mathcal{Q},L) = v(\gamma_1,\mathcal{Q}) \cdot v(\gamma_2,\mathcal{Q}) \cdot L^h+ O_{\gamma_1,\gamma_2,\mathcal{Q}}\left(L^{h-\kappa}\right). \qedhere
	\]
\end{theorem}

\begin{remark}
	The constant $\kappa = \kappa(\mathcal{Q}) > 0$ in Theorem \ref{theo:main} is related to the exponential mixing rate of the Teichmüller geodesic flow on $\mathcal{Q}$. For a precise definition of the constants $v(\gamma_1,\mathcal{Q}) > 0$ and $v(\gamma_2,\mathcal{Q}) > 0$ see \S 3.
\end{remark}

\subsection*{Main ideas of the proof.} To prove Theorem \ref{theo:main} we first recast the counting function $sq(\gamma_1,\gamma_2,\mathcal{Q},L)$ as a counting function of mapping class group orbits of integral simple closed multi-curves in train track coordinates. This is done using the connections first introduced in \cite{Ara20d}; see Proposition \ref{prop:recast}. We then apply a variant of the sophisticated theory for counting mapping class group orbits of integral simple closed multi-curves in train track coordinates developed in \cite{EMM19}; see Theorem \ref{theo:EMM}. This application requires a very careful handling of error terms.

Given a topological surface $S$ and an integral simple closed multi-curve $\gamma$ on $S$, an important step of the proof of Theorem \ref{theo:main} corresponds to parametrizing the quotient space $\mf(\gamma)/\mathrm{Stab}(\gamma)$ in terms of train track coordinates; this is the space of singular measured foliations that together with $\gamma$ fill $S$ modulo the stabilizer of $\gamma$ in the mapping class group. This parametrization is achieved by introducing the concept of moderately slanted cylinder diagrams and moderately slanted cylinder train tracks; see Proposition \ref{prop:train_tracks}.

\begin{remark}
	The ideas introduced in the proof of Theorem \ref{theo:main} can also be used to give effective estimates of other related counting functions of square-tiled surfaces; see Theorem \ref{theo:main_2} for an example.
\end{remark}

\subsection*{Outline of the paper.} In \S2 we introduce the concepts of moderately slanted cylinder diagrams and moderately slanted cylinder train tracks and use them to parametrize the quotient space $\mf(\gamma)/\mathrm{Stab}(\gamma)$. In \S3 we go into more detail on the ideas introduced in \cite{Ara20d} and \cite{EMM19} and use them prove Theorem \ref{theo:main} following the sketch described above.

\subsection*{Acknowledgments.} The author would like to thank Alex Eskin and Amir Mohammadi for enlightening conversarions. The author would also like to thank Alex Wright and Pouya Honaryar for comments on an earlier version of this paper. This work got started while the author was participating in the \textit{Dynamics: Topology and Numbers} trimester program at the Hausdorff Research Institute for Mathematics (HIM). The author is very grateful for the hospitality of the HIM and for the hard work of the organizers of the trimester program. This work was finished while the author was a member of the Institute for Advanced Study (IAS). The author is very grateful to the IAS for its hospitality. This material is based upon work supported by the National Science Foundation under Grant No. DMS-1926686.

\section{Horospheres in connected components of strata}

\subsection*{Outline.} In this section we discuss horospheres on connected components of strata of quadratic differentials from several points of view. After defining the horospherically foliated sets $\mathcal{Q}(\gamma)$ we parametrize them using a special class of cylinder diagrams. We then discuss how to parametrize them in terms of singular measured foliations and, more concretely, how to do so using train track coordinates. This last perspective will be particularly useful in the proof of Theorem \ref{theo:main}.

\subsection*{Horospheres.} For the rest of this section fix an integer vector $\sigma := (\sigma_1,\dots,\sigma_n)$ with $\sigma_i \geq -1$ and a boolean value $\varepsilon \in \{0,1\}$. Denote by $Q(\sigma,\varepsilon)$ the stratum of unmarked quadratic differentials with marked singularties of order $\sigma$ and $\varepsilon = 1$ if and only if every quadratic differential in the stratum is the square of an Abelian differential. A singularity of order $-1$ is a pole, a singularity of order $0$ is a marked point, and a singularity of order $\geq 1$ is a zero of the corresponding order. Let $g \geq 0$ be the non-negative integer satisfying the equation
\[
4g-4 = \sum_{i =1}^n \sigma_i.
\]
The complex dimension of the stratum $Q(\sigma,\varepsilon)$ is given by
\[
h:= 2g-2+n+\epsilon.
\]
For the rest of this section we fix a connected component $\mathcal{Q}$ of this stratum.

Let $S$ be a connected, oriented surface of genus $g$ with $n$ punctures. Denote by $\mathcal{Q}\mathcal{T}$ the lift of the connected component $Q$ of quadratic differentials marked by the surface $S$. Denote by $\mcg:= \mathrm{Mod}(S)$ the mapping class group of $S$. The action of $\mcg$ on $\mathcal{Q}\mathcal{T}$ is properly discontinuous and $Q$ is the corresponding orbifold quotient.

Denote by $\mathcal{MF} := \mathcal{MF}(S)$ the space of singular measured foliations on $S$. The markings on $Q\mathcal{T}$ allow us to define maps
\[
\Re, \Im \colon Q\mathcal{T} \to \mathcal{MF}
\]
that record the vertical and horizontal foliations, respectively, of any marked quadratic differential. Fix an integral simple closed multi-curve $\gamma := a_1\gamma_1+\dots+a_k\gamma_k$ on $S$. Any such multi-curve determines a horospherically foliated set
\[
\mathcal{Q}\mathcal{T}(\gamma) := \{q \in \mathcal{QT} \ | \ \Im(q) =  \gamma\}.
\]
The maps $\Re$ and $\Im$ above factor through the mapping class group action to give
\[
\boldsymbol{\Re}, \boldsymbol{\Im} \colon \mathcal{Q} \to \mathcal{MF}/\mcg. 
\]
Equivalence classes of integral simple closed multi-curves are called topological types. Fixing a topological type gives rise to a horospherically foliated set
\[
\mathcal{Q}(\gamma) := \{q \in \mathcal{Q} \ | \ \boldsymbol{\Im}(q) = \mcg\cdot\gamma\}.
\]

Denote by $\mathrm{Stab}(\gamma) \subseteq \mcg$ the set of mapping classes that fix $\gamma$. It will be convenient for our purposes to consider the intermediate quotient $\mathcal{QT}(\gamma)/ \mathrm{Stab}(\gamma)$. The following proposition will be particularly useful in this context.

\begin{proposition}\cite[Lemma 3.6]{Ara20d}
	\label{prop:int_quot}
	The quotient map $\mathcal{QT}(\gamma)/\mathrm{Stab}(\gamma) \to \mathcal{Q}(\gamma)$ is a homeomorphism.
\end{proposition}

\subsection*{Cylinder diagrams.} Quadratic differentials in $\mathcal{Q}(\gamma)$ can be represented via cylinder diagrams. A cylinder diagram is a collection of disjoint parallelograms on the complex plane with certain edge identifications. Each parallelogram has one pair of sides, called bases, which are parallel to the real axis. These bases are broken into edges and pairs of edges are identified by translation and/or 180\degree~rotation. The other pair of sides,  refered to as special edges, are always identified by translation. See Figure \ref{fig:cyl_diag} for an example.

\begin{figure}[h!]
	\centering
	\includegraphics[scale=0.7]{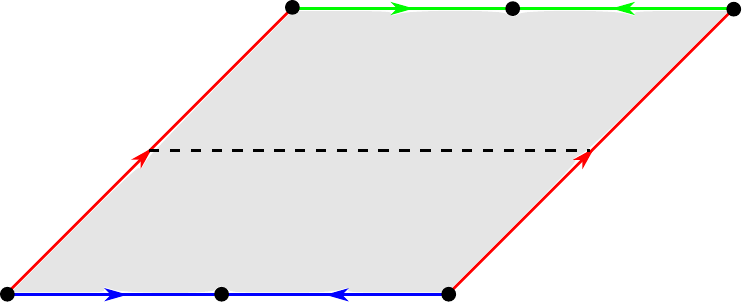}
	\caption{Cylinder diagram of a quadratic differential in $\mathcal{Q}(4,0)$. The horizontal foliation is identified with a simple closed curve on a genus $0$ surface with $4$ punctures that separates the punctures into two sets of two. This is a moderately slanted cylinder diagram.}
	\label{fig:cyl_diag}
\end{figure}

The horizontal foliation of a cylinder diagram is identified with a weighted simple closed multi-curve on the corresponding surface. The components of the multi-curve are given by the core curves of the cylinders and each such curve is weighted by the height of the corresponding cylinder. In the case of cylinder diagrams coming from integral simple closed multi-curves, the heights of the cylinders are always integers.

Each cylinder diagram represents a quadratic differential in a given connected component of a stratum. As we vary the real parts of the edges of a cylinder diagram while retaining the parallelism conditions, the corresponding quadratic differential remains in the same connected component. The identifications in the moduli space $\mathcal{Q}(\gamma)$ correspond to cut and paste operations.

\begin{proposition}
	\cite[\S3]{Ara20d}
	\label{prop:corresp_0}
	The set $\mathcal{Q}(\gamma)$ is in one-to-one correspondence with the set of cylinder diagrams in $\mathcal{Q}$ with horizontal foliation of type $\mcg \cdot \gamma$ up to cut-and-paste operations.
\end{proposition}

To get a more complete description of the correspondence in Proposition \ref{prop:corresp_0} we restrict our attention to a particular class of cylinder diagrams. Consider a cylinder diagram with pairs of bases of lengths $b_1,\dots,b_k > 0$ and special edges of holonomy with real parts $s_1,\dots,s_k \in \mathbf{R}$. We say this cylinder diagram is moderately slanted if $0 < 
s_i \leq b_i$ for every $i \in \{1,\dots,k\}$. See Figure \ref{fig:cyl_diag} for an example. The following proposition follows immediately by applying Dehn twists in an appropriate way.

%\begin{figure}[h!]
%	\centering
%	\includegraphics[scale=0.7]{pictures/slanted.pdf}
%	\caption{Cylinder diagram in Figure \ref{fig:cyl_diag} represented as a moderately slanted cylinder diagram.}
%	\label{fig:slanted}
%\end{figure}

\begin{proposition}
	Up to the action of $\mathrm{Stab}(\gamma)$, every cylinder diagram describing a quadratic differential in $\mathcal{Q}(\gamma)$ can be represented in a unique way as a moderately slanted cylinder diagram.
\end{proposition}

In particular, we deduce the following corollary.

\begin{corollary}
	\label{cor:cylinder_diag}
	There exists finitely many moderately slanted cylinder diagrams which, by varying the real parts of their edges while retaining the parallelism and moderately slanted conditions, represent all quadratic differentials in $\mathcal{Q}(\gamma)$ without overlaps.
\end{corollary}

\subsection*{Measured foliations.} Quadratic differentials in $\mathcal{Q}(\gamma)$ can also be parametrized in terms of their vertical foliations. More precisely, denote by $i(\cdot,\cdot)$ the geometric intersection number pairing on $\mf \times \mf$. Consider the set $\Delta \subseteq \mf \times \mf$ of pairs of non-filling singular measured foliations, i.e.,
\[
\Delta := \{(\lambda,\mu) \in \mf \times \mf \ | \ \exists \eta \in \mf \colon i(\lambda,\eta) + i(\mu,\eta) = 0 \}.
\]
By work of Hubbard and Masur \cite{HM79}, the map
\[
(\Re,\Im) \colon \mathcal{Q}\mathcal{T} \to \mf \times \mf \setminus \Delta
\]
is a $\mcg$-equivariant homeomorphism onto its image. Denote 
\[
\mf(\gamma) := \{\lambda \in \mf \ | \ (\gamma,\lambda) \notin \Delta\}.
\]

\begin{proposition}\cite{HM79}
	\label{prop:corresp1}
	The map $\Re \colon \mathcal{QT}(\gamma) \to \mf(\gamma)$ is a $\mathrm{Stab}(\gamma)$-equivariant homeomorphism onto its image.
\end{proposition}

Of particular importance for us will be the quotient $\mf(\gamma)/\mathrm{Stab}(\gamma)$. Directly from the Proposition \ref{prop:corresp1} we deduce the following corollary.

\begin{corollary}
	\label{cor:corresp2}
	The induced map $\Re \colon \mathcal{QT}(\gamma)/\mathrm{Stab}(\gamma) \to \mf(\gamma)/\mathrm{Stab}(\gamma)$ is a homeomorphism onto its image.
\end{corollary}

\subsection*{Train tracks.} The correspondence in Proposition \ref{prop:corresp1} and Corollary \ref{cor:corresp2} will allow us to study the quotient $\mf(\gamma)/\mathrm{Stab}(\gamma)$ using moderately slanted cylinder diagrams. We can make this idea more precise using train tracks. Given a moderately slanted cylinder diagram, consider a triangulation by saddle connections of the underlying surface $S$ as in Figure \ref{fig:triang}. 

\begin{figure}[h!]
	\centering
	\includegraphics[scale=0.7]{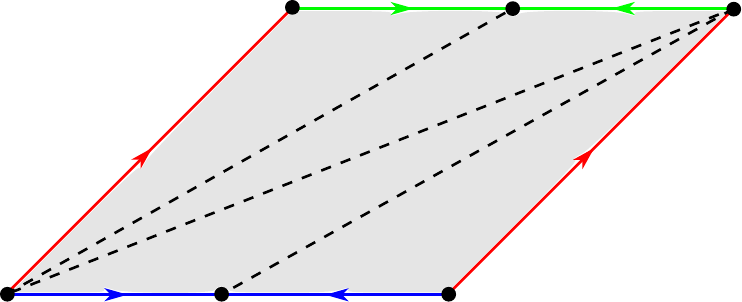}
	\caption{Triangulation associated to the moderately slanted cylinder diagram in Figure \ref{fig:cyl_diag}.}
	\label{fig:triang}
\end{figure}

On each of the triangles of this triangulation consider a 1-complex as in Figure \ref{fig:complex_0}; the edges of this complex that do not intersect the sides of the triangle will be referred to as inner edges. Label the edges of the triangle by $a, b, c$ so that 
\[
|\Re(a)| = |\Re(b)| + |\Re(c)|.
\]
The edge labeled $a$ is unique because the cylinder diagram is moderately slanted. Delete the inner edge of the complex of the triangle opposite to $a$ as in Figure \ref{fig:tt}. Joining these complexes along the edges of the triangulation as in Figure \ref{fig:tt_complete} yields a train train $\tau$ on $S$ that carries the singular measured foliation $\Re(q)$; the weights of the train track correspond to the absolute value of the real parts of the edges of  the triangulation. Furthermore, the area of $q$ is equal to $i(\Re(q),\gamma)$.
	
\begin{figure}[h]
	\centering
	\begin{subfigure}[b]{0.4\textwidth}
		\centering
		\includegraphics[width=0.7\textwidth]{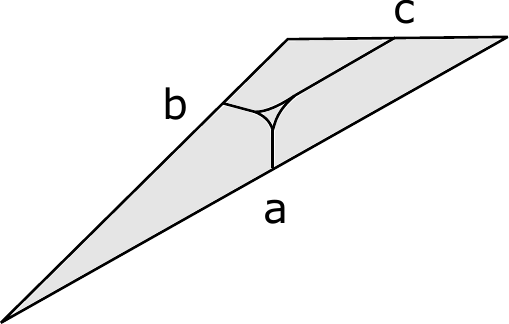}
		\caption{The initial 1-complex.}
		\label{fig:complex_0}
	\end{subfigure}
	\quad \quad \quad
	~ %add desired spacing between images, e. g. ~, \quad, \qquad, \hfill etc. 
	\begin{subfigure}[b]{0.4\textwidth}
		\centering
		\includegraphics[width=0.7\textwidth]{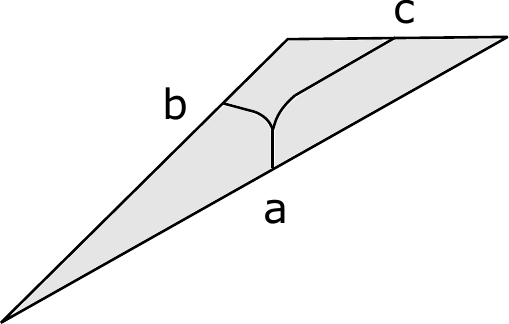}
		\caption{Deleting an inner edge.}
		\label{fig:tt}
	\end{subfigure}
	\caption{The $1$-complexes in a triangle.} 
	\label{fig:complex}
\end{figure}

\begin{figure}[h!]
	\centering
	\includegraphics[scale=0.7]{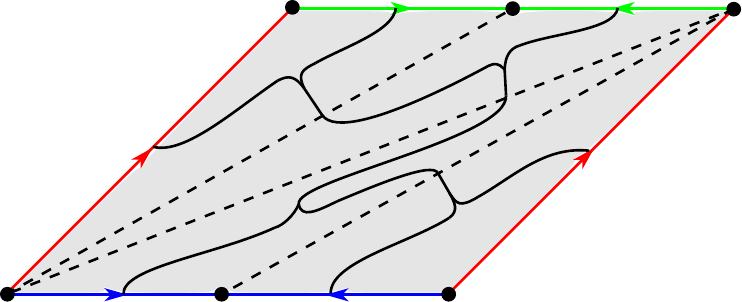}
	\caption{Train track associated to the moderately slanted cylinder diagram in Figure \ref{fig:cyl_diag}.}
	\label{fig:tt_complete}
\end{figure}

We refer to the train tracks constructed above as moderately slanted cylinder train tracks. Varying the real parts of the edges of the cylinder diagram while retaining the parallelism and moderately slanted conditions preserves the train track. Moreover, the moderately slanted condition corresponds to an explicit linear cone in the weight space of the train track. We refer to this cone as the \textit{characteristic cone} of the train track. Directly from this discussion, Proposition \ref{prop:int_quot}, Corollary \ref{cor:cylinder_diag}, and Corollary \ref{cor:corresp2}, we deduce the following result.

\begin{proposition}
	\label{prop:train_tracks}
	There exist finitely many moderately slanted cylinder train tracks which, in their characteristic cone, carry all singular measured foliations in the image of $\Re \colon \mathcal{Q}\mathcal{T}(\gamma) \to \mf(\gamma)$ up to the action of $\mathrm{Stab}(\gamma)$ without overlaps.
\end{proposition}

Recall $\gamma := a_1\gamma_1 + \dots + a_k \gamma_k$. Consider a moderately slanted cylinder diagram with $k$ cylinders $C_1,\dots,C_k$ representing a quadratic differential $q \in \mathcal{Q}(\gamma)$. Denote by $\tau$ the corresponding moderately slanted cylinder train track. Let $u_1,\dots,u_k > 0$ be the weights on $\tau$ corresponding to the special edges of each cylinder. For each cylinder $C_i$ let $v_1^{(i)},\dots,v_{m(i)}^{(i)} > 0$ be the weights on $\tau$ corresponding to the edges on the top base of the cylinder ordered from left to right. For each cylinder $C_i$ let $w_1^{(i)},\dots,w_{l(i)}^{(i)} > 0$ be the weights on $\tau$ corresponding to the edges on the bottom base of the cylinder ordered from left to right. The weights considered completely determine the rest of the weights of the $\tau$. Notice that
\[
\mathrm{Area}(q) := \frac{1}{2}\left(\sum_{i =1}^k a_i \left(\sum_{j=1}^{m(i)} v_j^{(i)} + \sum_{j=1}^{l(i)} w_j^{(i)}\right)\right).
\]

Directly from the discussion above we deduce the following.

\begin{proposition}
	\label{prop:area}
	In the weight space of a moderately slanted cylinder train track, the area of a quadratic differential, or, equivalently, the geometric intersection number with the horizontal foliation, is given by a linear functional whose sub-level-sets are bounded when intersected with the characteristic cone.
\end{proposition}

\section{Counting square-tiled surfaces}

\subsection*{Outline of this section.} In this section we prove Theorem \ref{theo:main}, the main result of this paper. The proof relies on the sophisticated theory developed in work of Eskin, Mirzakhani, and Mohammadi \cite{EMM19} for counting mapping class group orbits of simple closed multi-curves in train track coordinates. We begin by reviewing this theory and discussing some variants. We then apply this theory together with the result discussed in \S 2 to prove Theorem \ref{theo:main}.

\subsection*{Mapping class group orbits.} In \cite{EMM19}, Eskin, Mirzakhani, and Mohammadi proved an effective estimate for the number of simple closed curves of length $\leq L$ on a compact surface equipped with a Riemannian metric of negative curvature. To prove this result, a sophisticated theory for counting mapping class group orbits of integral simple closed multi-curves in train track coordinates was developed. We now summarize the main aspects of this theory as well as discuss some variants.

Let $S$ be a connected, oriented surface of genus $g$ with $n$ punctures and $\tau$ be a train track on $S$. Denote by $U(\tau)$ the cone of non-negative weights on $\tau$ satisfying the switch conditions and by $\|\cdot\|$ the $L^1$ norm on $U(\tau)$. Consider the set
\[
P(\tau) := \{\lambda \in U(\tau) \ | \ \|\lambda\| = 1\}.
\]
By a polyhedron $\mathcal{U} \subseteq P(\tau)$ we mean a polyhedron of dimension $\mathrm{dim} \thinspace U(\tau) -1$ where the number of facets and the angles are bounded below by uniform constants depending only on $S$; facets are allowed to be open and/or closed. For the rest of this discussion let $\gamma := a_1\gamma_1 + \cdots + a_k \gamma_k$ be an integral simple closed multi-curve on $S$. For every $L \geq 0$ consider the counting function
\[
s(\gamma,\mathcal{U},L) := \#\left\lbrace
\begin{array}{c | l}
\alpha \in \mathrm{Mod}(S) \cdot \gamma  & \alpha \in \mathbf{R}^+ \cdot \mathcal{U} \\
 & \|\alpha\| \leq L
\end{array}\right\rbrace.
\]

Let $\mathcal{Q}$ the principal stratum of quadratic differentials on $S$. Denote $v(\gamma) > 0$ the Lebesgue measure of the set of quadratic differentials $q \in Q(\gamma)$ with $\mathrm{Area}(q) \leq 1$; this quantity is finite because of Propositions \ref{prop:train_tracks} and \ref{prop:area}. Denote by $\mu_\mathrm{Thu}$ the Thurston measure on $\mf(S)$. Recall that $\mf(S)$ can be endowed with a natural $\mathbf{R}^+$ action that scales transverse measures.

\begin{theorem} \cite[Theorem 7.1]{EMM19}
	There exists $\kappa = \kappa(S) > 0$  such that for every maximal train track $\tau$ on $S$, every polyhedron $\mathcal{U} \subseteq U(\tau)$, and every $L \geq 0$,
	\[
	s(\gamma,\mathcal{U},L) = v(\gamma) \cdot \mu_\mathrm{Thu}\left((0,1]\cdot \mathcal{U}\right) \cdot L^{6g-6} + O_{\gamma,\tau}\left(L^{6g-6-\kappa}\right).
	\]
\end{theorem}

In the ensuing discussion we use the notation introduced in \S 2. Let $\mathcal{Q}$ be a connected component of a stratum of quadratic differentials. Recall that we denote by $h > 0$ its complex dimension. Denote by $v(\gamma,\mathcal{Q}) > 0$ the Lebesgue measure of the set of quadratic differentials $q \in Q(\gamma)$ with $\mathrm{Area}(q) \leq 1$; this quantity is finite because of Propositions \ref{prop:train_tracks} and \ref{prop:area}. Given a moderately slanted cylinder train track $\tau$ denote by $\mu$ the Lebesgue measure on its weight space.

\begin{theorem}
	\label{theo:EMM}
	There exists $\kappa = \kappa(\mathcal{Q}) > 0$  such that for every moderately slanted cylinder train track $\tau$ on $S$ carrying vertical foliations of quadratic differentials in  $\mathcal{Q}(\gamma)$, every polyhedron $\mathcal{U} \subseteq U(\tau)$, and every $L \geq 0$,
	\[
	s(\gamma,\mathcal{U},L) = v(\gamma,\mathcal{Q}) \cdot \mu\left((0,1]\cdot \mathcal{U}\right) \cdot L^h + O_{\gamma,\tau}\left(L^{h-\kappa}\right).
	\]
\end{theorem}

\begin{proof}
	The result follows by the same arguments used to prove \cite[Theorem 7.1]{EMM19}. More concretely, the main technical tool used in the proof of this result is the bisector counting estimate \cite[Proposition 4.1]{EMM19}. This estimate holds for any connected components of a stratum. Indeed, the backbone of the proof of this estimate is the horosphere equidistribution result \cite[Proposition 3.2]{EMM19}, which holds for every connected component of a stratum. The main driving force behind the proof of this equidistribution result is the exponential mixing rate of the Teichmüller geodesic flow. This flow is known to be mixing on any connected component of a stratum due to work of Avila, Gouëzel, Resende, and Yoccoz \cite{AGY06,AR12,AG13}. The rest of the technical results used in the proof of \cite[Theorem 7.1]{EMM19} also holds in the proposed setting.
\end{proof}

\subsection*{Square-tiled surfaces.} For the rest of this section fix $\mathcal{Q}$ a connected component of a stratum of quadratic differentials on a surface $S$ and integral simple closed multi-curves $\gamma_1$ and $\gamma_2$ on $S$. For every $L \geq 0$ consider the counting function
\[
sq(\gamma_1,\gamma_2,\mathcal{Q},L) := \#\left(\left\lbrace \begin{array}{c|l}
\text{square-tiled surfaces $q $} & q \in \mathcal{Q},\\
& \boldsymbol{\Re}(q) = \mcg \cdot \gamma_1, \\
& \boldsymbol{\Im}(q) = \mcg \cdot \gamma_2, \\
& \mathrm{Area}(q) \leq L. 
\end{array} \right\rbrace \bigg/ \sim\right),
\]
where $\sim$ denotes the equivalence relation induced by cut-and-paste operations.

Our goal for the rest of this section is to prove an effective estimate for this counting function. To do so we first recast it as a counting function of mapping class group orbits of integral simple closed multi-curves in train track coordinates. We then apply Theorem \ref{theo:EMM} to get the desired effective estimate.

\subsection*{Recasting.} By work of Hubbard and Masur \cite{HM79}, square-tiled surfaces are in one-to-one correspondence with filling pairs of integral simple closed multi-curves. Given a filling pair of integral simple closed multi-curves $\alpha$ and $\beta$ on $S$, denote by $q(\alpha,\beta)$ the corresponding square-tiled surface.  Notice that
\begin{equation*}
\label{eq:int}
\mathrm{Area}(q(\alpha,\beta)) = i(\alpha,\beta).
\end{equation*}
Using this correspondence we can recast the counting function $sq(\gamma_1,\gamma_2,\mathcal{Q},L)$ in a more convenient way; compare to \cite[\S 3]{Ara20d}.

\begin{proposition}
	\label{prop:recast}
In the context above, for every $L \geq 0$,
\[
sq(\gamma_1,\gamma_2,\mathcal{Q},L) =\#\left(\left\lbrace \begin{array}{c | l}
\alpha \in \mcg \cdot \gamma_1 & \alpha \in \mf(\gamma_2),\\
& q(\alpha,\gamma_2) \in \mathcal{Q}, \\
& i(\alpha,\gamma_2) \leq L.
\end{array} \right\rbrace \bigg/ \mathrm{Stab}(\gamma_2)\right).
\]
\end{proposition}

\begin{proof}
	Using the correspondence between square-tiled surfaces and filling pairs of integral simple closed multi-curves we write
	\[
	sq(\gamma_1,\gamma_2,\mathcal{Q},L) =\#\left(\left\lbrace \begin{array}{c | l}
	(\alpha,\beta) \in \mcg \cdot \gamma_1 \times \mcg \cdot \gamma_2  & (\alpha,\beta) \notin \Delta,\\
	& q(\alpha,\beta) \in \mathcal{Q}, \\
	& i(\alpha,\beta) \leq L.
	\end{array} \right\rbrace \bigg/ \mcg\right),
	\]
	where the quotient by $\mcg$ corresponds to the diagonal action on $\mf \times \mf$. In turn, this expression can be rewritten as
	\[
	sq(\gamma_1,\gamma_2,\mathcal{Q},L) =\#\left(\left\lbrace \begin{array}{c | l}
	\alpha \in \mcg \cdot \gamma_1 & \alpha \in \mf(\gamma_2),\\
	& q(\alpha,\gamma_2) \in \mathcal{Q}, \\
	& i(\alpha,\gamma_2) \leq L.
	\end{array} \right\rbrace \bigg/ \mathrm{Stab}(\gamma_2)\right). \qedhere
	\]
\end{proof}

\subsection*{Geometric intersection numbers.} Fix $\tau$ a moderately slanted cylinder train track carrying singular measured foliations in $\mf(\gamma_2)$ corresponding to quadratic differentials in $\mathcal{Q}(\gamma_2)$. Proposition \ref{prop:area} guarantees the function $i(\cdot,\gamma_2)$ is a linear function over the characteristic cone of $\tau$. In particular, this function is Lipschitz. Given a non-empty polyhedron $\mathcal{U} \subseteq P(\tau)$ in the characteristic cone of $\tau$ denote
\begin{gather*}
M(\gamma_2,\mathcal{U}) := \max_{\lambda \in \mathcal{U}} i(\lambda,\gamma_2), \\ m(\gamma_2,\mathcal{U}) := \min_{\lambda \in \mathcal{U}} i(\lambda,\gamma_2).
\end{gather*}
Proposition \ref{prop:area} ensures these quantities are positive and finite. Denote the $L^1$ diameter of the polyhedron $\mathcal{U}$ by $\mathrm{diam}(\mathcal{U})$. As a direct consequence of the discussion above we deduce the following.

\begin{proposition}
	\label{prop:lip}
	For every polyhedron $\mathcal{U}  \subseteq P(\tau)$ in the characteristic cone of $\tau$,
	\[
	|M(\gamma_2,\mathcal{U}) - m(\gamma_2,\mathcal{U})| \preceq_{\gamma_2,\tau} \mathrm{diam}(\mathcal{U}).
	\]
\end{proposition}

\subsection*{Comparison.} For every non-empty polyhedron $\mathcal{U} \subseteq P(\tau)$ in the characteristic cone of $\tau$ and every $L \geq 0$ consider the counting function
\begin{gather*}
sq(\gamma_1,\gamma_2,\mathcal{U},L) :=\#\left\lbrace \begin{array}{c | l}
\alpha \in \mcg \cdot \gamma_1 & \alpha \in \mathbf{R}_+ \cdot \mathcal{U}\\
& i(\alpha,\gamma_2) \leq L
\end{array} \right\rbrace.
\end{gather*}

When $\mathcal{U} \subseteq P(\tau)$ is the intersection of the characteristic cone of $\tau$ with $P(\tau)$ we denote this counting function simply by $s(\gamma_1,\gamma_2,\tau,L)$. By Proposition \ref{prop:recast}, a first step towards the proof of Theorem \ref{theo:main}, the main result of this paper, would be to study the counting functions $sq(\gamma_1,\gamma_2,\mathcal{U},L)$. Theorem \ref{theo:EMM} allows us to study the counting functions $s(\gamma_1,\mathcal{U},L)$. The following bounds, which follow directly from the definitions, will thus play an important role.

\begin{proposition}
	For every non-empty polyhedrom $\mathcal{U} \subseteq P(\tau)$ in the characteristic cone of $\tau$ and every $L \geq 0$, the following bounds hold,
	\begin{gather*}
	sq(\gamma_1,\gamma_2,\mathcal{U},L) \leq s(\gamma_1,\mathcal{U},L/m(\gamma_2,\mathcal{U})), \\
	 s(\gamma_1,\mathcal{U},L/M(\gamma_2,\mathcal{U})) \leq sq(\gamma_1,\gamma_2,\mathcal{U},L).
	\end{gather*}
\end{proposition}

Applying Theorem \ref{theo:EMM} we deduce the following corollary.

\begin{corollary}
	\label{cor:count}
	There exists a constant $\kappa = \kappa(\mathcal{Q}) > 0$ such that for every non-empty polyhedrom $\mathcal{U} \subseteq P(\tau)$ in the characteristic cone of $\tau$ and every $L \geq 0$, 
	\begin{gather*}
	sq(\gamma_1,\gamma_2,\mathcal{U},L) \leq v(\gamma_1,\mathcal{Q}) \cdot \mu\left((0,1]\cdot \mathcal{U}\right) \cdot (L/m(\gamma_2,\mathcal{U}))^h + O_{\gamma_1,\tau}\left(L^{h-\kappa}\right) , \\
	v(\gamma_1,\mathcal{Q}) \cdot \mu\left((0,1]\cdot \mathcal{U}\right) \cdot (L/M(\gamma_2,\mathcal{U}))^h + O_{\gamma_1,\tau}\left(L^{h-\kappa}\right)  \leq sq(\gamma_1,\gamma_2,\mathcal{U},L).
	\end{gather*}
\end{corollary}

\subsection*{Leading terms.} Let $\mathcal{V} \subseteq P(\tau)$ be the intersection of the characteristic cone of $\tau$ with $P(\tau)$. Given a finite partition $\mathcal{U} := \{\mathcal{U}_i\}_{i=1}^N$ of $\mathcal{V}$, denote 
\[
\mathrm{diam}(\mathcal{U}) := \max_{i \in \{1,\dots,N\}} \mathrm{diam}(\mathcal{U}_i).
\]
Denote by $v(\gamma_2,\tau) > 0$ the positive constant
\[
v(\gamma_2,\tau) := \mu \left(\{\lambda \in \mathbf{R}^+ \cdot \mathcal{V} \ | \ i(\lambda,\gamma_2) \leq 1 \} \right).
\]

The following proposition will allow us to identify the leading term of the counting function $s(\gamma_1,\gamma_2,\tau,L)$ in the estimates that will be carried out later.

\begin{proposition}
	\label{prop:lead}
	For every partition $\mathcal{U} := \{\mathcal{U}_i\}_{i=1}^N$ of $\mathcal{V}$,
	\begin{gather*}
	\sum_{i=1}^N \frac{\mu\left((0,1]\cdot \mathcal{U}_i\right)}{M(\gamma_2,\mathcal{U}_i)^h} \leq v(\gamma_2,\tau) \leq \sum_{i=1}^N \frac{\mu\left((0,1]\cdot \mathcal{U}_i\right)}{m(\gamma_2,\mathcal{U}_i)^h}, \\
	\sum_{i=1}^N \frac{\mu\left((0,1]\cdot \mathcal{U}_i\right)}{m(\gamma_2,\mathcal{U}_i)^h} - \sum_{i=1}^N \frac{\mu\left((0,1]\cdot \mathcal{U}_i\right)}{M(\gamma_2,\mathcal{U}_i)^h} \preceq_{\gamma_2,\tau} \mathrm{diam}(\mathcal{U}).
	\end{gather*}
\end{proposition}

\begin{proof}
	The first set of bounds follows directly from the definitions and the fact that the measure $\mu$ is $h$-homogeneous under positive scalings. For the second bound notice that, by applying Proposition \ref{prop:lip},
	\begin{align*}
	\sum_{i=1}^N \frac{\mu\left((0,1]\cdot \mathcal{U}\right)}{m(\gamma_2,\mathcal{U})^h} &- \sum_{i=1}^N \frac{\mu\left((0,1]\cdot \mathcal{U}\right)}{M(\gamma_2,\mathcal{U})^h}\\
	 &\preceq_{\gamma_2,\tau} \mu((0,1] \cdot \mathcal{V}) \cdot \max_{i \in \{1,\dots,N\}}|M(\gamma_2,\mathcal{U}_i) - m(\gamma_2,\mathcal{U}_i)|\\
	&\preceq_{\gamma_2,\tau} \mathrm{diam}(\mathcal{U}).\qedhere
	\end{align*}
\end{proof}

\subsection*{Characteristic cones.} We are now ready to prove an effective estimate for the counting function $sq(\gamma_1,\gamma_2,\tau,L)$. This will be the main tool used in the proof of Theorem \ref{theo:main}, the main result of this paper.

\begin{proposition}
	\label{prop:count_pre}
	There exists a constant $\kappa = \kappa(\mathcal{Q}) > 0$ such that for every $L \geq 0$, 
	\[
	sq(\gamma_1,\gamma_2,\tau,L) = v(\gamma_1,\mathcal{Q}) \cdot v(\gamma_2,\tau) \cdot L^h + O_{\gamma_1,\gamma_2,\tau}\left(L^{h-\kappa}\right).
	\]
\end{proposition}

\begin{proof}
	Let $\delta \in (0,1)$ to be fixed later. Consider a partition $\mathcal{U} := \{\mathcal{U}_i\}_{i=1}^N$  of diameter $\mathrm{diam}(\mathcal{U}) \leq \delta$ of the characteristic cone of $\tau$ into $N \preceq_{\tau} \delta^{-h}$ polyhedrons. By Corollary \ref{cor:count}, for each of these polyhedrons we have the estimates
	\begin{gather*}
	sq(\gamma_1,\gamma_2,\mathcal{U}_i,L) \leq v(\gamma_1,\mathcal{Q}) \cdot \mu\left((0,1]\cdot \mathcal{U}_i\right) \cdot (L/m(\gamma_2,\mathcal{U}_i))^h + O_{\gamma_1,\tau}\left(L^{h-\kappa}\right) , \\
	v(\gamma_1,\mathcal{Q}) \cdot \mu\left((0,1]\cdot \mathcal{U}_i\right) \cdot (L/M(\gamma_2,\mathcal{U}_i))^h + O_{\gamma_1,\tau}\left(L^{h-\kappa}\right)  \leq sq(\gamma_1,\gamma_2,\mathcal{U}_i,L).
	\end{gather*}
	Adding up these estimates over $i \in \{1,\dots,N\}$ we get
	\begin{gather*}
	sq(\gamma_1,\gamma_2,\tau,L) \leq v(\gamma_1,\mathcal{Q}) \cdot \sum_{i=1}^N \frac{\mu\left((0,1]\cdot \mathcal{U}_i\right)}{m(\gamma_2,\mathcal{U}_i)^h} \cdot L^h + O_{\gamma_1,\tau}\left(N \cdot L^{h-\kappa}\right) , \\
	v(\gamma_1,\mathcal{Q}) \cdot \sum_{i=1}^N \frac{\mu\left((0,1]\cdot \mathcal{U}_i\right)}{M(\gamma_2,\mathcal{U}_i)^h} \cdot L^h + O_{\gamma_1,\tau}\left(N \cdot L^{h-\kappa}\right)  \leq sq(\gamma_1,\gamma_2,\mathcal{U}_i,L).
	\end{gather*}
	By Proposition \ref{prop:lead}, it follows that
	\[
	sq(\gamma_1,\gamma_2,\tau,L) = v(\gamma_1,\mathcal{Q}) \cdot v(\gamma_2,\tau) \cdot L^h+ O_{\gamma_1,\gamma_2,\tau}\left(\mathrm{diam}(\mathcal{U}) \cdot L^h + N \cdot L^{h-\kappa} \right).
	\]
	Notice that
	\[
	\mathrm{diam}(\mathcal{U}) \cdot L^h + N \cdot L^{h-\kappa} \preceq_{\tau} \delta \cdot L^h + \delta^{-h} \cdot L^{h-\kappa}.
	\]
	Let $\delta := L^{-\eta}$ with $\eta > 0$. Choose $\eta > 0$ so that 
	\[
	\kappa'=\kappa'(\mathcal{Q}) := \min\{\eta, \kappa - h\eta\} > 0.
	\]
	It follows that
	\[
	sq(\gamma_1,\gamma_2,\tau,L) = v(\gamma_1,\mathcal{Q}) \cdot v(\gamma_2,\tau) \cdot L^h+ O_{\gamma_1,\gamma_2,\tau}\left(L^{h-\kappa'}\right). \qedhere
	\]
\end{proof}

\subsection*{Proof of the main result.} We are now ready to prove Theorem \ref{theo:main}.

\begin{proof}[Proof of Theorem \ref{theo:main}]
	By Proposition \ref{prop:train_tracks} there exists a finite collection $\{\tau_i\}_{i=1}^N$ of moderately slanted cylinder train tracks which in their characteristic cone carry all singular measured foliation in the image of $\Re \colon \mathcal{QT}(\gamma_2) \to \mf(\gamma_2)$ up to the action of $\mathrm{Stab}(\gamma_2)$ without overlaps. By Proposition \ref{prop:recast}, we have
	\[
	sq(\gamma_1,\gamma_2,\mathcal{Q},L) = \sum_{i=1}^N sq(\gamma_1,\gamma_2,\tau_i,L).
	\]
	By Proposition \ref{prop:count_pre}, for every $i \in \{1,\dots,N\}$,
	\[
	sq(\gamma_1,\gamma_2,\tau_i,L) = v(\gamma_1,\mathcal{Q}) \cdot v(\gamma_2,\tau_i) \cdot L^h+ O_{\gamma_1,\gamma_2,\tau_i}\left(L^{h-\kappa}\right).
	\]
	Adding up these estimates we conclude
	\[
	sq(\gamma_1,\gamma_2,\mathcal{Q},L) = v(\gamma_1,\mathcal{Q}) \cdot v(\gamma_2,\mathcal{Q}) \cdot L^h+ O_{\gamma_1,\gamma_2,\mathcal{Q}}\left(L^{h-\kappa}\right). \qedhere
	\]
\end{proof}

\subsection*{Further remarks.} As explained in \S1, the ideas introduced in the proof of Theorem \ref{theo:main} can also be used to give effective estimates of other related counting functions of square-tiled surfaces. More explicitly, for every $L \geq 0$ consider the counting function
\[
sq(\gamma_1,*,\mathcal{Q},L) := \#\left(\left\lbrace \begin{array}{c | l}
\text{square-tiled surfaces $q$} & q \in \mathcal{Q}, \\
& \boldsymbol{\Re}(q) = \mcg \cdot \gamma_1, \\
& \mathrm{Area}(q) \leq L.
\end{array} \right\rbrace \bigg/ \sim \right),
\]
where $\sim$ denotes the equivalence relation induced by cut and paste operations. 

Using standard lattice point counting arguments in place of Theorem \ref{theo:EMM} in the proof of Theorem \ref{theo:main} yields the following result.

\begin{theorem}
	\label{theo:main_2}
	There exists a constant $\kappa = \kappa(\mathcal{Q}) > 0$ such that for every $L \geq 0$,
	\[
	sq(\gamma_1,*,\mathcal{Q},L) = v(\gamma_1,\mathcal{Q}) \cdot L^h+ O_{\gamma_1,\mathcal{Q}}\left(L^{h-\kappa}\right).
	\]
\end{theorem}


\bibliographystyle{amsalpha}


\bibliography{bibliography}

\end{document}